\documentclass{article}
\usepackage{latexsym,amssymb}

\newcommand{\dar}{{\downarrow}}

\newtheorem{proposition}{Proposition}[section]
\newtheorem{lemma}[proposition]{Lemma}
\newtheorem{corollary}[proposition]{Corollary}
\newtheorem{definition}[proposition]{Definition}
\newtheorem{theorem}[proposition]{Theorem}

\newcounter{opgaveteller}
\newcounter{lijst-teller}

\newcounter{boon}
\newcounter{boon2}
\newenvironment{proof}{\noindent {\bf Proof}. \nopagebreak }{\nopagebreak\hfill\rule{2mm}{3mm}}

\input xypic
\xyoption{all}
\UseComputerModernTips
\CompileMatrices
\definemorphism{dat}\dashed\tip\notip
\definemorphism{dub}\Solid\Tip\notip

\newenvironment{rlist}%
   {\begin{list}{\roman{lijst-teller}\/)\hfil}%
              {\labelwidth 2em%
               \leftmargin\labelwidth\advance\leftmargin by\labelsep%
               \usecounter{lijst-teller}}}%
   {\end{list}}

   {\end{list}}

\newenvironment{arlist}
    {\begin{list}{\arabic{boon2}\/)\hfil}%
                 {\labelwidth 2em%
                 \leftmargin\labelwidth\advance\leftmargin by\labelsep%
                 \usecounter{boon2}}}%
    {\end{list}}

\title{Partial Combinatory Algebras of Functions}
\author{Jaap van Oosten\\ Department of Mathematics\\ Utrecht University\\ P.O.Box 
80.010, 3508 TA Utrecht, The Netherlands\\ {\tt J.vanOosten@uu.nl}}
\date{May 13, 2009}

\begin{document}
\maketitle

\begin{abstract} 
\noindent We employ the notions of `sequential function' and `interrogation' (dialogue) in order to define new partial combinatory algebra structures on sets of functions. These structures are analyzed using J. Longley's preorder-enriched category of partial combinatory algebras and decidable applicative structures.

We also investigate total combinatory algebras of partial functions. One of the results is, that every realizability topos is a quotient of a realizability topos on a total combinatory algebra.\end{abstract}


\section*{Introduction}\label{introsection}
Let us think of a computing device which, in the course of its calculations, is allowed to consult an oracle. I wish to keep the intuition of `computing device' as flexible as possible and refrain therefore from a definition; but one requirement I want to stick to is: the device will use only finitely many oracle queries in any terminating computation (there may be never terminating computations in which the device just keeps on passing queries to the oracle).

If a terminating computation always results in an output, the device then determines a partial function ${\cal O}\stackrel{\Phi}{\to}{\cal R}$, where $\cal O$ is the set of oracles and $\cal R$ the set of results (outputs). In cases of practical interest, $\cal R$ is a discrete set such as the set $\mathbb{N}$ of natural numbers, whereas $\cal O$, like the set of all functions $\mathbb{N}\to\mathbb{N}$, has a nontrivial topology. The finiteness requirement above implies in this example that the partial function ${\cal O}\to {\cal R}$ is continuous, and this is often taken as the meaning of the Use Principle in Recursion Theory (e.g.\  {\cite{SoareR:recesd}}, p.50: ``The Use Principle asserts that $\Phi _e$ is continuous'').

In this paper I concentrate on the situation where $\cal O$ is the set $A^A$ of functions $A\to A$ for some infinite set $A$, and queries are of form: ``what is your value at $a\in A$?''. I argue that here it makes sense to consider a subclass of the class of (partial) continuous functions on $A^A$: the {\em sequential\/} functions. Actually, the map $\Phi$ will turn out to be sequential.

It is shown that the sequential functions play an important role in the construction of {\em partial combinatory algebras}. I show that Kleene's construction of a partial combinatory algebra structure on $\mathbb{N}^{\mathbb{N}}$ (\cite{KleeneSC:forrff}; see \cite{OostenJ:reaics}, 1.4.3, for a concise exposition) can be generalized to a partial combinatory algebra structure on any set $A^A$ for infinite $A$.

Then, some further analysis is carried out in the case that $A$ itself has a partial combinatory algebra structure and the coding which is necessary for defining the structure of $A^A$, is actually definable from $A$. A universal property is obtained in Longley's category (\cite{LongleyJ:reatls}) of partial combinatory algebras and decidable applicative morphisms. We also look at sub-partial combinatory algebras of $A^A$.

Next, we discuss partial combinatory algebras on sets of {\em partial\/} functions. In the end we obtain a theorem in the theory of realizability toposes: every realizability topos is a quotient of a realizability topos on a total combinatory algebra.

\section{Sequential (partial) functions}\label{seqsection}
The notion `sequential' has been around for a long time, and stems from the study of Plotkin's calculus {\sc PCF} (\cite{PlotkinG:lcfcpl}). For a fairly recent paper discussing various approaches to the matter, see \cite{mellies:seqass}. 

The notion of a sequential tree was defined in \cite{OostenJ:casfft}. I slightly modify it here.
\begin{definition}\label{seqtreedef}\em Let $A$ be a set, and $T$ a set of finite functions $p:A'\to A$ with $A'\subset A$. We shall also write ${\rm dom}(p)$ for $A'$. The set $T$ is ordered by inclusion. $T$ is called a {\em sequential tree\/} if it is a tree, with the empty function as root, and has the property that for every $p\in T$ which is not a leaf, there is an element $a\not\in {\rm dom}(p)$ such that for all immediate successors $q$ of $p$ in $T$, we have ${\rm dom}(q)={\rm dom}(p)\cup\{ a\}$.

A sequential tree is {\em total\/} if for each $p\in T$ which is not a leaf, there is $a\not\in {\rm dom}(p)$ such that the set of immediate successors of $p$ in $T$ is the set of {\em all\/} functions $q$ satisfying ${\rm dom}(q)={\rm dom}(p)\cup\{ a\}$.\end{definition}
We shall mainly be interested in total sequential trees. Clearly, for such a tree, every function $A\to A$ determines a path through the tree. Suppose $F$ is a function from the set of leaves of $T$ to $A$. Then $F$ and $T$ determine a partial function $\Phi _{T,F}:A^A\to A$ as follows: $\Phi _{T,F}(\alpha )$ is defined if and only if the path through $T$ determined by $\alpha$ ends in a leaf $v$; and in that case, $\Phi _{T,F}(\alpha )=F(v)$.
\begin{definition}\label{seqfundef}\em A partial function $A^A\to A$ of the form $\Phi _{T,F}$ is called a {\em partial sequential function}.\end{definition}
Note, that a function $\Phi _{T,F}$ is a total function (i.e., everywhere defined) if and only if the tree $T$ is well-founded. 

The set $A$ is given the discrete topology and the set of functions $A^A$ the product topology. One of the lemmas underlying the construction of Kleene's ${\cal K}_2$ (a partial combinatory algebra of functions $\mathbb{N}^{\mathbb{N}}$) is that when $A$ is countable, every partial continuous function $A^A\to A$ with open domain is partial sequential.

If $A$ is uncountable, we still have that every {\em total\/} continuous function $A^A\to A$ is sequential, but this may fail for partial functions with open domain, as the following proposition shows.
\begin{proposition}\label{cont=seq} Let $A$ be an infinite set. Every continuous function $A^A\to A$ is sequential, but if $A$ is uncountable, there exist partial continuous functions with open domain, that are not partial sequential.\end{proposition}
\begin{proof} Let $\Phi :A^A\to A$ be continuous. For a finite function $p:A'\to A$ with $A'\subset A$, let $U_p\, =\, \{ \alpha\in A^A\, |\, p\subset\alpha \}$ be the open neighborhood determined by $p$. By continuity, the function $\Phi$ has a {\em base\/} $\cal B$, that is: a set of finite functions $p$ such that $\Phi$ is constant on $U_p$ and such that for every $\alpha\in A^A$ there is a $p\in {\cal B}$ such that $p\subset\alpha$.

Call two finite functions $s$ and $t$ {\em compatible\/} if $s\cup t$ is a function (i.e., $s(a)=t(a)$ whenever $a\in {\rm dom}(s)\cap {\rm dom}(t)$). For arbitrary finite $s$, let ${\cal B}_s$ be the set of those $p\in {\cal B}$ such that $p$ and $s$ are compatible.
\medskip

\noindent {\sl Claim} For any finite function $s$, either $\Phi$ is constant on $U_s$ or there is a finite subset $C$ of $A-{\rm dom}(s)$ such that for every $p\in {\cal B}_s$, ${\rm dom}(p)$ meets $C$.

\noindent {\sl Proof of Claim}: Suppose there is no such $C$; then given any two elements $p,q$ of ${\cal B}_s$ we can find $r\in {\cal B}_s$ such that ${\rm dom}(r)-{\rm dom}(s)$ is disjoint from $({\rm dom}(p)\cup {\rm dom}(q))-{\rm dom}(s)$. Then $r$ is compatible both with $p$ and with $q$, and since $\Phi$ is constant on $U_p$, $U_q$ and $U_r$, $\Phi$ takes the same values on $U_p$ and $U_q$. We conclude that $\Phi$ is constant on $U_s$.
\medskip

We now build a sequential tree for $\Phi$ as follows: $T$ will be the union of a sequence $T_0\subseteq T_1\subseteq\cdots$ of well-founded sequential trees.

Let $T_0$ consist of only the empty function. Suppose $T_n$ has been defined. For every leaf $s$ of $T_n$ define the set of elements of $T_{n+1}$ extending $s$ as follows: if $\Phi$ is constant on $U_s$, this set is empty (and $s$ is also a leaf of $T_{n+1}$). Otherwise, pick a finite set $C$ for $s$ as in the {\sl Claim}; and order $C$ as $\{ c_1,\ldots ,c_n\}$. Then add for each $k$, $1\leq k\leq m$, all functions extending $s$ whose domain is ${\rm dom}(s)\cup\{ c_1,\ldots ,c_k\}$.

Each $T_n$ is clearly a well-founded sequential tree, by induction; and by construction the following is true: if $p\in {\cal B}$ and $p$ is compatible with a leaf $s$ of $T_n$, then either $\Phi$ is constant on $U_s$ (in which case $s$ is a leaf of $T$), or the cardinality of $s\cap p$ is at least $n$. Hence, since the sets $\{ U_p\, |\, p\in {\cal B}\}$ cover $A^A$, the tree $T$ is well-founded. Define for any leaf $s$ of $T$, $F(s)=\Phi (\alpha )$ for an arbitrary $\alpha\in U_s$. Then $\Phi = \Phi _{T,F}$.
\medskip

\noindent For the second statement, split $A$ in two disjoint, nonempty subsets $A_0$, $A_1$. Let $\Phi$ be a partial function which is constant on its domain, and which is only defined on those $\alpha$ for which there is an $a\in A$ with $\alpha (a)\in A_0$. Then $\Phi$ cannot be partial sequential: suppose it is; let $T$ be a sequential tree for it. Choose $\alpha$ such that $\alpha (a)\not\in A_0$ for each $a\in A$. Then the path through $T$ determined by $\alpha$ must be infinite by assumption on $T$, but the union of this path is a partial function on $A$ with countable domain $A'$. But then any function $\alpha '$ which agrees with $\alpha$ on $A'$ but has $\alpha '(a)\in A_0$ for some $a\not\in A'$, will determine the same infinite path, although $\Phi (\alpha ')$ should be defined.\end{proof}
\medskip

\noindent Partial sequential functions $A^A\to A$ can be coded by elements of $A^A$ in the following way. Let $A^*$ be the free monoid on $A$, i.e.\ the set of finite sequences of elements of $A$. Fix an injective function
$$(a_0,\ldots ,a_{n-1})\;\mapsto\; \langle a_0,\ldots ,a_{n-1}\rangle\; :A^*\to A$$
The elements in the image of this map are called {\em coded sequences}.
Let $q$ and $r$ (for `query' and `result') be two specified, distinct elements of $A$. With these data we define a partial operation $\varphi :A^A\times A^A\to A$ as follows.

For $\alpha ,\beta\in A^A$ and $u=\langle a_0,\ldots ,a_{n-1}\rangle$ a coded sequence, call $u$ an {\em interrogation\/} of $\beta$ by $\alpha$, if for each $j\leq n-1$ there is an $a\in A$ such that $\alpha (\langle a_0,\ldots ,a_{j-1}\rangle )=\langle q,a\rangle$ and $\beta (a)=a_j$. Of course, for $j=0$ this means that $\alpha (\langle\,\rangle )=\langle q,a\rangle$ and $\beta (a)=a_0$. The elements $\langle q,a\rangle$ are called the {\em queries\/} of the interrogation.\footnote{In \cite{OostenJ:casfft}, I called these ``dialogues''. Now I find they are far too one-sided for that name}

Note that $\alpha$ and $\beta$ uniquely determine a sequence of interrogations (the {\em interrogation process}) which may be finite or infinite. We shall apply the notion of interrogation also to finite functions.

We say that $\varphi (\alpha ,\beta )$ {\em is defined with value\/} $b$, if there is an interrogation $u$ of $\beta$ by $\alpha$ such that $\alpha (u)=\langle r,b\rangle$. We call the element $\langle r,b\rangle$ the {\em result\/} of the interrogation $u$.

Write $\varphi _{\alpha}$ for the partial function $\beta\mapsto\varphi (\alpha ,\beta ): A^A\to A$.
\begin{proposition}\label{seqcoding} A partial function $A^A\to A$ is of the form $\varphi _{\alpha}$ for some $\alpha\in A^A$, precisely when it is sequential.\end{proposition}
\begin{proof} Suppose we have a partial sequential function $\Phi _{T,F}$. Then for any $s\in T$, the sequential tree structure of $T$ induces a linear order on ${\rm dom}(s)$, say ${\rm dom}(s)=\{ c_0,\ldots ,c_{n-1}\}$ (and the predecessors of $s$ in $T$ are restrictions of $s$ to subsets $\{ c_0,\ldots ,c_{j-1}\}$). Let $a_i=s(c_i)$. Let $\alpha$ be any function $A\to A$ such that for each $s\in T$, with $c_i,a_i$ as above, $\alpha (\langle a_0,\ldots ,a_{n-1}\rangle )=\langle r,F(s)\rangle$ if $s$ is a leaf of $T$, and $\alpha (\langle a_0,\ldots ,a_{n-1}\rangle )=\langle q,c_n\rangle$ if $c_n$ is the unique element of ${\rm dom}(t)-{\rm dom}(s)$ for each immediate successor $t$ of $s$ in $T$. Clearly then, $\varphi _{\alpha}=\Phi _{T,F}$.

Conversely, given $\alpha$, let $T_0$ be the set of all finite functions $s$ such that there exists an interrogation of $s$ by $\alpha$ which contains all the values of $s$. It is easy to see that $T_0$ is a total sequential tree. Let us look at the leaves of $T_0$. If $s$ is such a leaf, we have the following 3 possibilities:\begin{arlist}
\item the interrogation process of $s$ by $\alpha$ is infinite: $\alpha$ continues to ask for information it has already received;
\item for some interrogation $u$ of $s$ by $\alpha$, $\alpha (u)$ is neither a query nor a result;
\item for some interrogation $u$ of $s$ by $\alpha$, $\alpha (u)=\langle r,b\rangle$ for some $b$.\end{arlist}
Let $T$ result from $T_0$ by the following: for every leaf $s$ of $T_0$ for which 1) or 2) holds, choose an injective function $(a_0,a_1,\ldots )$ of $\mathbb{N}$ into $A-{\rm dom}(s)$, and add all finite functions of the form $s\cup t$ for which ${\rm dom}(t)=\{ a_0,\ldots ,a_k\}$ for some $k\geq 0$.

Finally, if $s$ is a leaf of $T$ (hence a leaf of $T_0$ to which 3) applies), define $F(s)=b$ if $\alpha (u)=\langle r,b\rangle$ for the shortest interrogation $u$ of $s$ by $\alpha$ yielding a result. Then $\varphi _{\alpha}=\Phi _{T,F}$.\end{proof}
\section{A partial combinatory algebra structure on $A^A$}\label{pcasection}
Our aim is to prove that with the partial map $\alpha ,\beta\mapsto\alpha\beta$ defined below in definition~\ref{applicationdef}, the set $A^A$ has the structure of a {\em partial combinatory algebra}. Let us first recall what this means:
\begin{definition}\label{pcadef}\em A {\em partial combinatory algebra\/} is a set $X$ together with a partial function $X\times X\to X$, written $x,y\mapsto xy$, such that there exist elements $k$ and $s$ in $X$ satisfying the two axioms:\begin{itemize}
\item[$(k)$] For any $x,y\in X$, $kx$ and $(kx)y$ are defined, and $(kx)y=x$;
\item[$(s)$] For any $x,y,z\in X$, $sx$ and $(sx)y$ are defined, and $((sx)y)z$ is defined precisely if $(xz)(yz)$ is, and these expressions define the same element of $X$ if defined.\end{itemize}
The partial function $a,b\mapsto ab$ is called the {\em application function}.\end{definition}
\begin{definition}\label{applicationdef}\em For $\alpha ,\beta\in A^A$ and $a\in A$, call $u=\langle a_0,\ldots ,a_{n-1}\rangle$ an $a$-{\em interrogation\/} of $\beta$ by $\alpha$, if for each $j\leq n-1$, there is a $b\in A$ such that $\alpha (\langle a,a_0,\ldots ,a_{j-1}\rangle )=\langle q,b\rangle$ and $\beta (b)=a_j$.

We say that $\varphi ^a(\alpha ,\beta )$ is defined with value $c$, if for some $a$-interrogation $u$ of $\beta$ by $\alpha$, $\alpha (u)=\langle r,c\rangle$.

Then, define a partial function $A^A\to A$, denoted $\alpha\beta\mapsto \alpha\beta$ in the following way: $\alpha\beta$ is defined if and only if for every $a\in A$, $\varphi ^a(\alpha ,\beta )$ is defined. In that case, $\alpha\beta$ is the function $a\mapsto\varphi ^a(\alpha ,\beta )$.\end{definition}
The following proposition is straightforward.
\begin{proposition}\label{appl=seq}For any $\alpha$ and $a$, the partial function $\beta\mapsto\varphi ^a(\alpha ,\beta )$ is sequential.

Conversely, suppose that for every $a\in A$ we are given a partial sequential function $F_a:A^A\to A$. Then there is an element $\alpha _F$ of $A^A$ such that for all $\beta\in A^A$ and all $a\in A$, $\varphi ^a(\alpha ,\beta )$ is defined if and only if $F_a(\beta )$ is, and equal to it in that case. Hence if for all $a$ $F_a(\beta )$ is defined, $\alpha\beta$ is defined and for all $a$, $\alpha\beta (a)=F_a(\beta )$.\end{proposition}
We shall have to deal with sequential functions of more than one variable.
\begin{definition}\label{biseqdef}\em Let $T$ be a set of pairs of finite functions $(s,t)$, ordered by pairwise inclusion. We say that $T$ is a {\em bisequential tree\/} if $T$ is a tree, and for any non-leaf $(s,t)\in T$ we have: {\em either\/} there is $a\not\in {\rm dom}(s)$ such that the set of immediate successors of $(s,t)$ in $T$ is the set of finite functions $(s',t)$ where $s'$ extends $s$ and ${\rm dom}(s')={\rm dom}(s)\cup\{ a\}$, {\em or\/} there is $b\not\in {\rm dom}(t)$ such that the set of immediate successors is the set of $(s,t')$ with ${\rm dom}(t')={\rm dom}(t)\cup\{ b\}$.

Any pair of functions $A\to A$ determines a unique path through a bisequential tree, and just as for the case of one variable we say that a partial function $\phi :A^A\times A^A\to A$ is bisequential if there is a bisequential tree $T$ and a function $F$ from the leaves of $T$ to $A$, such that $\phi =\Phi _{T,F}$.
\end{definition}
\begin{lemma}\label{bisequential=representable} Let $G$ be a total bisequential function $A^A\times A^A\to A$. Then there is an element $\phi _G$ of $A^A$ such that for all $\alpha$ and $\beta$, $\phi _G\alpha$ is defined, and $\varphi (\phi _G\alpha ,\beta ) =G(\alpha ,\beta )$.\end{lemma}
\begin{proof} Let $G$ be $\Phi _{T,F}$, so $T$ is a bisequential tree, $F$ a function from the leaves of $T$ to $A$, such that $G(\alpha ,\beta )=F((p,q))$ for the unique leaf $(p,q)$ determined by $(\alpha ,\beta )$. Note, that since $G$ is total, the tree $T$ is well-founded. 

Call a non-leaf $(s,t)$ of $T$ a $(0,u)$-point if all immediate successors of $(s,t)$ are of form $(s',t)$ with ${\rm dom}(s')={\rm dom}(s)\cup\{ u\}$; similarly, a $(1,v)$-point has immediate successors $(s,t')$ with ${\rm dom}(t')={\rm dom}(t)\cup\{ v\}$. Suppose $s(u_0)=a_0,\ldots ,s(u_{n-1})=a_{n-1}$, $t(v_0)=b_0,\ldots ,t(v_{m-1}=b_{m-1}$ are the values of $s$ and $t$ in the path, in that order. We define the value of $\phi _G$ on the interrogation $$\langle\langle b_0,\ldots ,b_{m-1}\rangle ,a_0,\ldots ,a_{n-1}\rangle \; =\; \phi _G(\langle\langle\vec{b}\rangle ,\vec{a}\rangle )$$
If $(s,t)$ is a $(0,u)$-point, let $\phi _G(\langle\langle\vec{b}\rangle ,\vec{a}\rangle )=\langle q,u\rangle$; if $(s,t)$ is a $(1,v)$-point, let $\phi _G(\langle\langle\vec{b}\rangle ,\vec{a}\rangle )=\langle r,\langle q,v\rangle\rangle$. Finally, if $(s,t)$ is a leaf, $\phi _G(\langle\langle\vec{b}\rangle ,\vec{a}\rangle )=\langle r,\langle r,F(s,t)\rangle$.

It is then straightforward to verify that for every $(\alpha ,\beta )$ passing through the point $(s,t)$, we have for all $j\leq m-1$ that $(\phi _G\alpha )(\langle b_0,\ldots ,b_{j-1}\rangle )=\langle q,v_j\rangle$, and if $(s,t)$ is a leaf of $T$, we have $(\phi _G\alpha )(\langle b_0,\ldots ,b_{m-1}\rangle )=\langle r, G(\alpha ,\beta )$. Since $T$ is well-founded, it is easy to complete the definition of $\phi _G$ in such a way that $\phi _G\alpha$ is always defined. Then also, $\varphi (\phi _G\alpha ,\beta )=G(\alpha ,\beta )$ as desired. \end{proof}
\begin{corollary}\label{prutcorollary} Suppose for each $a\in A$ a total bisequential function $G_a:A^A\times A^A\to A$ is given. Then there is an element $\phi _G$ of $A^A$ such that for all $\alpha ,\beta\in A^A$ and $a\in A$, $\phi _G\alpha$ is defined, $(\phi _G\alpha )\beta$ is defined and $((\phi _G\alpha )\beta )(a)=G_a(\alpha ,\beta )$.\end{corollary}
\begin{proof} Straightforward from propositions~\ref{appl=seq} and \ref{bisequential=representable}.\end{proof}
\medskip

\noindent We can now state the main theorem of this section.
\begin{theorem}\label{A^Apca} For an infinite set $A$, $A^A$, together with the map $(\alpha ,\beta )\mapsto\alpha\beta$, has the structure of a partial combinatory algebra.\end{theorem}
\begin{proof} We have to find elements $k$ and $s$ satisfying $(k)$ and $(s)$ of definition~\ref{pcadef}.

Since for any $a\in A$ the map $(\alpha ,\beta )\to\alpha (a)$ is bisequential, it follows at once from corollary~\ref{prutcorollary} that there is an element $k$ of $A^A$ such that $(k\alpha )\beta =\alpha$.

For $s$, we have to do a bit more work. Let $\alpha $, $\beta$ be fixed for the moment. We define a function $S^{\alpha\beta}$ as follows: we define recursively the values of $S^{\alpha\beta}$ on elements of the form 
$$\langle a\rangle\ast u\; =\; \langle a,a_0,\ldots ,a_{m-1}\rangle$$
of which we assume, inductively, that $u=\langle a_0,\ldots ,a_{m-1}\rangle$ is an $a$-interrogation of a finite function $t$ by $S^{\alpha\beta}$.

Assume the interrogation $u$ has length $n$. Determine a maximal sequence
$$(v_0^0,\ldots ,v_0^{n_0-1},b_0,w_0^0,\ldots ,w_0^{m_0-1},c_0,\ldots ,v_j^0,\ldots ,v_j^{n_j-1},b_j,w_j^0,\ldots ,w_j^{m_j-1},c_j,\ldots )$$
of length $\leq n$, such that for any segment
$$(v_j^0,\ldots ,v_j^{n_j-1},b_j,w_j^0,\ldots ,w_j^{m_j-1},c_j)$$
or initial parts of it, the following holds (where applicable):\begin{rlist}
\item $\langle v_j^0,\ldots ,v_j^{n_j-1}\rangle$ is an $\langle a,c_0,\ldots ,c_{j-1}\rangle$-interrogation of $t$ by $\alpha$ with result $\langle q,b_j\rangle$ (so the value of $\alpha$ on this sequence is $\langle r,\langle q,b_j\rangle\rangle$);
\item $\langle w_j^0,\ldots ,w_j^{m_j-1}\rangle$ is a $b_j$-interrogation of $t$ by $\beta$ with result $c_j$.\end{rlist}
We define the value $S^{\alpha\beta} (\langle a\rangle\ast u)$ as follows:\begin{arlist}
\item if the sequence ends in $(v_j^0,\ldots ,v_j^k)$ and
$$\alpha (\langle\langle a,c_0,\ldots ,c_{j-1}\rangle ,v_j^0,\ldots ,v_j^k\rangle )\; =\; \langle q,x\rangle$$
then $S^{\alpha\beta}(\langle a\rangle\ast u)=\langle q,x\rangle$;
\item if the sequence ends in $(b_j,w_j^0,\ldots ,w_j^k)$ and
$$\beta (\langle b_jw_j^0,\ldots ,w_j^k\rangle )\; =\;\langle q,x\rangle$$
then $S^{\alpha\beta}(\langle a\rangle\ast u)=\langle q,x\rangle$;
\item if the sequence ends in $(v_j^0,\ldots ,v_j^{n_j-1})$ and
$$\alpha (\langle\langle a,c_0,\ldots ,c_{j-1}\rangle ,v_j^0,\ldots ,v_j^{n_j-1}\rangle )\; =\;\langle r,\langle r,y\rangle\rangle$$
then $S^{\alpha\beta}(\langle a\rangle\ast u)=\langle r,y\rangle$;
\item in all other cases, $\alpha (\langle a\rangle\ast u)=\langle q,q\rangle$.\end{arlist}
Now it is a matter of straightforward verification that if $\gamma$ is any function extending $t$, and $\alpha\gamma$, $\beta\gamma$ are defined, then the sequence $c_0,\ldots ,c_j$ forms an $a$-interrogation of $\beta\gamma$ by $\alpha\gamma$.

Hence, $(S^{\alpha\beta}\gamma )(a)$ is defined precisely when $((\alpha\gamma )(\beta\gamma ))(a)$ is, and equal to it in that case.

It is also left to you to check by inspection of the definition of $S^{\alpha\beta}$, that for fixed $a$ and $u$, the function
$$(\alpha ,\beta )\mapsto S^{\alpha\beta}(\langle a\rangle\ast u)$$
is bisequential.

By corollary~\ref{prutcorollary} it follows that there is an element $s\in A^A$ such that for all $\alpha$ and $\beta$, $s\alpha$ and $(s\alpha )\beta$ are defined, and $(s\alpha )\beta =S^{\alpha\beta}$. Then by the remarks above, this $s$ satisfies axiom $(s)$ of definition~\ref{pcadef}. We conclude that $A^A$, with the given partial map $(\alpha ,\beta )\mapsto\alpha\beta$, is a partial combinatory algebra, as claimed.

\end{proof}
\medskip

\noindent We shall denote the partial combinatory algebra on $A^A$ by ${\cal K}_2(A)$.
\section{Further analysis of ${\cal K}_2(A)$}\label{furthersection}
In this section we try to analyze the construction of ${\cal K}_2(A)$ a bit, from the point of view of Longley's 2-category PCA of partial combinatory algebras (first defined in \cite{LongleyJ:reatls}; there is also a description in \cite{OostenJ:reaics}).
\medskip

\noindent {\bf Convention}. From now on, when dealing with iterated applications we shall use the familiar convention of `associating to the left': i.e.\ we write $abcd$ instead of $((ab)c)d$.

PCA is a preorder-enriched category. The objects are partial combinatory algebras. Given two such, $A$ and $B$, a 1-cell, or {\em applicative morphism}, from $A$ to $B$ is a total relation $\gamma$ from $A$ to $B$ (we think of $\gamma$ as a function $A\to {\cal P}^*(B)$ into the set of nonempty subsets of $B$), with the property that there exists an element $r\in B$ such that, whenever $a,a'\in A$, $b\in\gamma (a),b'\in\gamma (a')$ and $aa'$ is defined, then $rbb'$ is defined and an element of $\gamma (aa')$. The element $r$ is called a {\em realizer\/} for $\gamma$.

If $\gamma ,\gamma ':A\to B$ are two applicative morphisms, we say $\gamma\preceq\gamma '$ if there is an element $s\in B$ such that for all $a\in A$ and all $b\in\gamma (a)$, $sb$ is defined and an element of $\gamma '(a)$. The element $s$ is said to {\em realize\/} $\gamma\preceq\gamma '$. For two parallel arrows $\gamma ,\gamma ':A\to B$ we write $\gamma\cong\gamma '$ if $\gamma\preceq\gamma '$ and $\gamma '\preceq\gamma$.

It is part of the theory of partial combinatory algebras that every partial combinatory algebra $A$ contains elements $\bot$, $\top$ and $C$ (thought of as `Booleans' and `definition by cases'), satisfying for all $a,b\in A$:
$$C\top ab=a\;\;\text{and}\;\; C\bot ab=b$$
Instead of $Cvab$ we write ${\sf If}\, v\, {\sf then}\, a\, {\sf else}\, b$.

Suppose $\gamma :A\to B$ is an applicative morphism and $\top _A,\bot _A$ are Booleans in $A$, $\top _B,\bot _B$ are Booleans in $B$. We call the morphism $\gamma$ {\em decidable\/} if there is an element $d\in B$ (the {\em decider\/} for $\gamma$) such that for all $b\in\gamma (\top _A)$, $db=\top _B$ and for all $b\in \gamma (\bot _A)$, $db=\bot _B$. There is a subcategory of PCA on the decidable applicative morphisms.

One further definition: if $\gamma :A\to B$ is an applicative morphism and $f$ is a partial function $A\to A$, then $f$ is said to be {\em representable\/} w.r.t.\ $\gamma$, if there is an element $r_f\in B$ (which then {\em represents\/} $f$), such that for all $a\in {\rm dom}(f)$ and all $b\in\gamma (a)$, $r_fb$ is defined and an element of $\gamma (f(a))$.
\begin{proposition}\label{AtoA^A} For $a\in A$ let $\hat{a}$ denote the constant function with value $a$. For any partial combinatory algebra structure on $A$, the map $\gamma (a)=\{\hat{a}\}$ defines a decidable applicative morphism $A\to {\cal K}_2(A)$. Every total function $A\to A$ is representable w.r.t.\ $\gamma$.\end{proposition}
\begin{proof} This is easy. Let $\rho$ be any element of $A^A$ satisfying:
$$\begin{array}{rcl}
\rho (\langle\langle x\rangle\rangle ) & = & \langle r,\langle q,q\rangle\rangle \\
\rho (\langle\langle x,b\rangle\rangle ) & = & \langle q,q\rangle \\
\rho (\langle\langle x,b\rangle ,a\rangle ) & = & \left\{\begin{array}{cl} \langle r,\langle r,ab\rangle\rangle & \text{if $ab$ is defined in $A$}\\ \langle r,\langle q,q\rangle\rangle & \text{otherwise}\end{array}\right. \\
\rho (\langle n\rangle ) & = & \langle r,r\rangle\text{ if $n\neq\langle x\rangle$ and $n\neq\langle x,b\rangle$}\end{array}$$
Then it is easily verified that if $ab$ is defined in $A$, $\rho\hat{a}$ is defined and $\rho\hat{a}\hat{b}=\widehat{ab}$. Hence, $\gamma$ is an applicative morphism. Furthermore, for any good choice of Booleans $\top ,\bot$ in $A$ one can take $\hat{\top},\hat{\bot}$ for Booleans in $A^A$, so decidability is easy. That every $f:A\to A$ is representable, is left to you.\end{proof}
\medskip

\noindent At this point I wish to collect a few bits of notation and theory of partial combinatory algebras; everything can be found in sections 1.1 and 1.3 of \cite{OostenJ:reaics}. Let $A$ be a partial combinatory algebra.\begin{arlist}
\item If $t$ and $s$ are terms built up from elements of $A$ and the application function, $t\dar$ means ``$t$ is defined'', and $t\simeq s$ means: $t$ is defined if and only if $s$ is, and they denote the same element of $A$ if defined.
\item Given a term $t(x_1,\ldots ,x_{n+1})$ built up from elements of $A$, variables $x_1,\ldots ,x_{n+1}$ and the application function, there is a standard construction for an element $\langle x_1\cdots x_{n+1}\rangle t$ of $A$ which satisfies:
$$\begin{array}{l}(\langle x_1\cdots x_{n+1}\rangle t)a_1\cdots a_n\dar \\
(\langle x_1\cdots x_{n+1}\rangle t)a_1\cdots a_{n+1}\,\simeq\, t(a_1,\ldots ,a_{n+1})\end{array}$$
\item $A$ has elements $p, p_0, p_1$ ({\em pairing\/} and {\em unpairing\/} combinators) such that $pab\dar$, $p_0(pab)=a$ and $p_1(pab)=b$; $A$ contains a copy $\{\bar{0},\bar{1},\ldots\}$ of the natural numbers such that any computable function is represented by an element of $A$; and $A$ has a standard coding of tuples $[\cdot ,\ldots ,\cdot ]$ together with elements representing the basic manipulation of these.
\item A has a {\em fixed-point operator}: an element $z$ such that for all $f,x\in A$: $zf\dar$ and $zfx\simeq f(zf)x$ (this is also referred to as ``the recursion theorem in $A$'').\end{arlist}
It is clear that the construction of ${\cal K}_2(A)$ given above, depends on the coding of tuples $\langle\cdot ,\ldots ,\cdot\rangle$ and the elements $q$ and $r$. Since we wish to study the connection to $A$ in the case $A$ itself has the structure of a partial combinatory algebra, we make the following definition.
\begin{definition}\label{basedonA}\em Suppose $A$ is an infinite set and $(A,\cdot )$ a partial combinatory algebra structure on $A$. We say that ${\cal K}_2(A)$ is {\em based on\/} $(A,\cdot )$ if, in the definition of an interrogation of $\beta$ by $\alpha$, we have used the standard coding $[\cdot ,\ldots ,\cdot ]$ of $A$, $q$ and $r$ are, respectively, the Booleans $\bot$ and $\top$, and the values of $\alpha$ at such interrogations are $p\bot u$ or $p\top u$.

We say that ${\cal K}_2(A)$ is {\em compatible with\/} $(A,\cdot )$ if there are elements $a,b,c\in A$ such that:\begin{rlist}\item for every tuple $u_0,\ldots ,u_{n-1}$, $a(\langle u_0,\ldots ,u_{n-1}\rangle )=[u_0,\ldots ,u_{n-1}]$ and\\ $b([u_0,\ldots ,u_{n-1}])=\langle u_0,\ldots ,u_{n-1}\rangle$;
\item $cq=\bot$ and $cr=\top$.\end{rlist}\end{definition} 
\begin{theorem}\label{factortheorem} Suppose $(A,\cdot )$ is a partial combinatory algebra and ${\cal K}_2(A)$ is based on $(A,\cdot )$. Let $\gamma :A\to {\cal K}_2(A)$ be the applicative morphism of~\ref{AtoA^A}. Then for any decidable applicative morphism $\delta :A\to B$ such that every total function $A\to A$ is representable w.r.t.\ $\delta$, there is a greatest decidable applicative morphism $\varepsilon :{\cal K}_2(A)\to B$ such that $\varepsilon\gamma\cong\delta$. Here, `greatest' and `$\cong$' refer to the preorder on applicative morphisms.\end{theorem}
\begin{proof} Given $\delta$, define $\varepsilon$ as follows: $\varepsilon (\alpha )$ is the (nonempty) set of elements $b\in B$ which represent $\alpha$ w.r.t.\ $\delta$.

First we show that $\varepsilon$ is an applicative morphism: we have to construct a realizer for $\varepsilon$. Let $r$ be a realizer for $\delta$ and $d$ a decider for $\delta$. Let $p,p_0,p_1$ in $A$ be the pairing and unpairing combinators, and $[\cdot ,\ldots ,\cdot ]$ the standard coding of tuples in $A$. Choose $\pi\in\delta (p)$ and $\pi _i\in\delta (p_i)$, for $i=0,1$. Let $c$ be an element of $B$ such that, if $u=[u_0,\ldots ,u_{k-1}]$ in $A$ and $y\in A$, $v\in\delta (u)$, $x\in\delta (y)$, then $cvx\in\delta ([u_0,\ldots ,u_{k-1},y])$. Let $s\in B$ be such that if $y\in A$ and $x\in\delta (y)$, then $sx\in\delta ([y])$.

Using the fixed-point combinator in $B$, find $F\in B$ satisfying for all $a,b,v\in B$: $Fab\dar$ and
$$Fabv\simeq\begin{array}{l}\mbox{{\sf If} }d(r\pi _0(av))\mbox{ {\sf then} }r\pi _1(av)\mbox{ {\sf else} }\\
Fab(cv(rb(r\pi _1(av)))) \end{array}$$
Now suppose $a\in\varepsilon (\alpha )$, $b\in\varepsilon (\beta )$.
\medskip

\noindent {\sl Claim}. For any $y\in A$, $x\in\delta (y)$, and any $y$-interrogation $u=[u_0,\ldots ,u_{k-1}]$ of $\beta$ by $\alpha$, there is a $v\in\delta ([y,u_0,\ldots ,u_{k-1}])$ such that $Fab (sx)\simeq Fabv$.
\medskip

\noindent This claim is proved by induction on $k$. For $k=0$, since $sx\in\delta ([y])$ there is nothing to prove.

Suppose the Claim holds for $j\leq k$, and $[u_0,\ldots ,u_{k}]$ is a $y$-interrogation. By induction hypothesis there is a $v\in\delta ([y,u_0,\ldots ,u_{k-1}])$ such that $Fab(sx)\simeq Fabv$. Since $[u_0,\ldots ,u_k]$ is a $y$-interrogation of $\beta$ by $\alpha$, we have $\alpha ([y,u_0,\ldots ,u_{k-1}])=p\bot e$ and $\beta (e)=u_k$, for some $e\in A$. Since $av\in\delta(\alpha ([y,u_0,\ldots ,u_{k-1}]))=\delta (p\bot e)$ we have $r\pi _0(av)\in\delta (\bot )$ so $d(r\pi _0(av))=\bot$ in $B$. By definition of $F$, we have
$$Fab(sx)\simeq Fabv\simeq Fab(cv(rb(r\pi _1(av))))$$
It is easily checked that $cv(rb(r\pi _1(av)))$ is an element of $\delta ([y,u_0,\ldots ,u_k])$. This proves the Claim.
\medskip

\noindent Now if $u=[u_0,\ldots ,u_{k-1}]$ is a $y$-interrogation of $\beta$ by $\alpha$ with result $g$, that is to say $\alpha ([y,u_0,\ldots ,u_{k-1}])=p\top g$, and $v\in\delta ([y,u_0,\ldots ,u_{k-1}])$ is as in the Claim, then by definition of $F$,
$$Fabv=r\pi _1(av)\in\delta (g)$$
We conclude that if $\alpha\beta (y)=g$ then $Fab(sx)\in\delta (g)$; hence, if $\alpha\beta\dar$, then $\langle x\rangle Fab(sx)\in\varepsilon (\alpha\beta )$. Therefore, the element $\rho =\langle abx\rangle Fab(sx)$ is a realizer for $\varepsilon$.
\medskip

\noindent That $\varepsilon$ is decidable follows easily from the fact that $\delta$ is, and the fact that in ${\cal K}_2(A)$ we may take $\hat{\bot}$ and $\hat{\top}$ for the Booleans.
\medskip

\noindent If $b\in B$ is an element of $\varepsilon\gamma (a)$, that is, $b$ represents $\hat{a}$ w.r.t.\ $\delta$, then for any chosen, fixed $\xi\in\bigcup_{a'\in A}\delta (a')$ we have $b\xi\in\delta (a)$ so $\langle b\rangle b\xi$ realizes $\varepsilon\gamma\preceq\delta$; conversely if $b\in\delta (a)$ then the element $\langle x\rangle b\in B$ clearly represents $\hat{a}$ w.r.t.\ $\delta$. So we see that $\varepsilon\gamma\cong\delta$.
\medskip

\noindent In order to see that $\varepsilon $ is the {\em greatest\/} applicative morphism satisfying $\varepsilon\gamma\cong\delta$, suppose $\varepsilon '$ is another one. Suppose $r'$ realizes $\varepsilon '$, $s$ realizes that $\varepsilon '\gamma\preceq\delta$, and $t$ realizes that $\delta\preceq\varepsilon\gamma$. In ${\cal K}_2(A)$ there is an element $\sigma$ such that for all $\alpha\in A^A$ and $a\in A$, $\sigma\alpha\hat{a}=\widehat{\alpha (a)}$ (this is left to the reader). Choose $\tau\in\varepsilon '(\sigma )$.

Let $\alpha\in A^A$ and $a\in A$ be arbitrary. Suppose $z\in\varepsilon '(\alpha )$. If $x\in\delta (a)$ then $tx\in\varepsilon '(\gamma (a))=\varepsilon '(\hat{a})$, so $r'(r'\tau z)(tx)\in\varepsilon '(\widehat{\alpha (a)})=\varepsilon '(\gamma (\alpha (a)))$; hence $s(r'(r'\tau z)(tx))\in\delta (\alpha (a))$.

We conclude that $\langle x\rangle s(r'(r'\tau z)(tx))$ represents $\alpha$ w.r.t.\ $\delta$; in other words, is an element of $\varepsilon (\alpha )$. Therefore, $\langle zx\rangle s(r'(r'\tau z)(tx))$ realizes $\varepsilon '\preceq\varepsilon$, as was to prove.\end{proof}
\medskip

\noindent Of course, theorem~\ref{factortheorem} also works if ${\cal K}_2(A)$ is compatible with $A$.

\subsection{Sub-pcas of ${\cal K}_2(A)$}\label{subpcasubsection}
We now turn our attention to sub-partial combinatory algebras of ${\cal K}_2(A)$: subsets $B\subset A^A$ such that, whenever $\alpha ,\beta\in B$ and $\alpha\beta\dar$ in ${\cal K}_2(A)$, then $\alpha\beta\in B$, and moreover, $B$ with the inherited partial application function is a partial combinatory algebra. For brevity, let's call $B$ a sub-pca of ${\cal K}_2(A)$.

A stronger notion, which is relevant to relative realizability (see \cite{BirkedalL:relmrr,BirkedalL:devttc-entcs}), requires $B$ to contain elements $k$ and $s$ which satisfy the axioms $(k)$ and $(s)$ of \ref{pcadef} both with respect to $B$ and with respect to ${\cal K}_2$. We call such sub-pcas {\em elementary}. Examples of elementary sub-pcas are: the inclusion of Rec in ${\cal K}_2$, where Rec is the set of total recursive functions; or $\Delta _n\subset {\cal K}_2$; or ${\rm RE}\subset {\cal P}(\omega )$, where RE is the set of recursively enumerable subsets of $\mathbb{N}$.

An instrument for studying sub-pcas of ${\cal K}_2(A)$, in the case $A$ has a partial combinatory algebra structure and ${\cal K}_2(A)$ is compatible with $A$) is the preorder $\leq _T$ on partial functions $A\to A$, defined in \cite{OostenJ:genfrr}. There, the following theorem is proved:
\begin{theorem}\label{A[f]} Let $A$ be a partial combinatory algebra and $f:A\to A$ a partial function. There is a partial combinatory algebra $A[f]$ and a decidable applicative morphism $\iota _f:A\to A[f]$ such that $f$ is representable w.r.t.\ $\iota _f$, and moreover any decidable applicative morphism $\gamma :A\to B$ such that $f$ is representable w.r.t.\ $\gamma$, factors uniquely through $\iota _f$.\end{theorem}
One can then define, for two partial functions $f,g:A\to A$: $f\leq _T g$ if $f$ is representable w.r.t.\ $\iota _g$. This gives a preorder on the set of partial endofunctions on $A$, which in the case that $A$ is ${\cal K}_1$ (the partial combinatory algebra of indices of partial recursive functions) and $f$ and $g$ are total functions, coincides with Turing reducibility.

Moreover, $A[f]$ is defined as follows. The underlying set is $A$ itself, and one defines a ``$b$-interrogation of $f$ by $a$'' just as in the definition of ${\cal K}_2(A)$ above, but now using application in $A$: i.e. it is a coded sequence $u=[u_0,\ldots ,u_{n-1}]$ such that for each $j<n$ there is a $v\in A$ such that $a([u_0,\ldots ,u_{j-1}])=[\bot ,v]$ and $f(v)=u_j$. Then $a\cdot ^fb=c$ if there is a $b$-interrogation $u$ of $f$ by $a$, such that $au=[\top ,c]$. The partial map $a,b\mapsto a\cdot ^fb$ is the application function for $A[f]$.

We see that if in ${\cal K}_2(A)$ the element $\alpha$ is representable in $A$, by $a\in A$, and $\alpha\beta$ is defined, then $\alpha\beta (x)=a\cdot ^{\beta}x$ for every $x\in A$. We see that $\alpha\beta$ is representable in $A[\beta ]$, so $\alpha\beta\leq _T\beta$.

We are led to conjecture that a sub-pca of ${\cal K}_2(A)$ should be {\em downwards closed\/} w.r.t.\ the preorder $\leq _T$. Let us see what can be said about this.
\begin{proposition}\label{subpca} Let $A$ be a partial combinatory algebra and suppose ${\cal K}_2(A)$ is compatible with $A$. Suppose $B\subset A^A$ is nonempty and closed under the application function (if $\alpha ,\beta\in B$ and $\alpha\beta\dar$, then $\alpha\beta\in B$).\begin{rlist}
\item if for every $a\in A$ there is an element $\alpha\in B$ which extends the partial function represented by $a$, then $B$ is downwards closed w.r.t.\ $\leq _T$;
\item without the hypothesis of i), the result may fail, even if $B$ is a sub-pca of ${\cal K}_2(A)$.\end{rlist}\end{proposition}
\begin{proof}  i) Suppose $\gamma\in B$ and $\beta \leq _T\gamma$. Then there is an $a\in A$ such that for all $x\in A$, $a\cdot ^{\gamma}x=\beta (x)$. If $\alpha\in B$ extends the partial function represented by $a$, we have $\alpha\gamma =\beta$. So $\beta\in B$, and $B$ is downwards closed w.r.t.\ $\leq _T$.

ii) My counterexample is a (non-elementary) sub-pca of Kleene's original ${\cal K}_2$. It is easiest to formulate in the original definition of ${\cal K}_2$: for $\alpha ,\beta\in\mathbb{N}^{\mathbb{N}}$ we say $\alpha\beta (x)=y$ if there is an $n$ such that
$$\begin{array}{rcll} \alpha\langle x,\beta (0),\ldots ,\beta (n-1)\rangle & = & y+1 & \\
\alpha\langle x,\beta (0),\ldots ,\beta (j-1)\rangle & = & 0 & \text{for all $j<n$}\end{array}$$
Then $\alpha\beta$ is defined if for all $x$ there is a $y$ such that $\alpha\beta (x)=y$.

Now define: $E_0=\langle\,\rangle$; $E_{n+1}=\langle E_n\rangle$. Let $B\subset\mathbb{N}^{\mathbb{N}}$ be given by $$B\; =\; \{\alpha\in\mathbb{N}^{\mathbb{N}}\, |\,\text{for all $n$, }\alpha (E_n)=n\}$$
It is easy to check that if $\alpha\in B$ and $\alpha\beta$ is defined in ${\cal K}_2$, then $\alpha\beta\in B$. Moreover, if $k'$ and $s'$ are the functions in $B$ which agree with $k$ (and $s$, respectively) outside $\{ E_0,E_1,\ldots\}$, then $k'$ and $s'$ satisfy the axioms $(k)$ and $(s)$ of \ref{pcadef}. So $B$ is a sub-pca of ${\cal K}_2$, but evidently not closed under `recursive in'.\end{proof}
\medskip

\noindent {\bf Remarks}\begin{arlist}\item The result of part i) in the proposition above can be strengthened a bit. In ordinary recursion theory, the poset of Turing degrees is a join-semilattice. It is not clear whether this is so for the general notion of $\leq _T$ considered here (but see section~\ref{partialsection}, but one can define the following: for $\alpha ,\gamma _1,\ldots ,\gamma _n\in A^A$ say $\alpha\leq _T(\gamma _1,\ldots ,\gamma _n)$ if $\alpha$ is representable in $A[\gamma _1,\ldots ,\gamma _n]$. Call $B\subset A^A$ an {\em ideal\/} if whenever $\gamma _1,\ldots ,\gamma _n\in B$ and $\alpha\leq _T(\gamma _1,\ldots ,\gamma _n)$, then $\alpha\in B$. One can prove that if $B$ satisfies the hypothesis of ii), then $B$ is an ideal.
\item About part ii): I do not know whether there exist elementary sub-pcas of ${\cal K}_2$ that are not downwards closed w.r.t.\ $\leq _T$.
\item I would have liked to include a statement in Proposition~\ref{subpca} saying that if $B$ is downwards closed w.r.t.\ $\leq _T$, then $B$ is an elementary sub-pca of ${\cal K}_2(A)$. The intuitive reason being, that $k$ and $s$ are definable in $A$, hence $\leq _T\beta$ for every $\beta\in B$, hence in $B$. However, this fails in general, because of the need of making $k$ and $s$ total. We have had to define $k$ and $s$ also outside the relevant interrogations. But in a general partial combinatory algebra it is not decidable whether or not a given element is a pair, or a coded sequence. We shall see that this problem disappears when we consider partial combinatory algebras of partial functions in section~\ref{partialsection}.\end{arlist} 

\noindent Theorem~\ref{factortheorem} can be generalized to certain sub-pcas of ${\cal K}_2(A)$. The proof is straightforward.
\begin{proposition}\label{factorsub} Suppose $B\subset A^A$ is a sub-pca which contains all constant functions $\hat{a}$ for $a\in A$ and the function $\rho$ from the proof of \ref{AtoA^A}. Then there is a decidable applicative morphism $\gamma :A\to B$ which has the property that whenever $\delta :A\to C$ is decidable and every element of $B$ is representable w.r.t.\ $\delta$, then there is $\varepsilon :B\to C$ such that $\varepsilon\gamma\cong\delta$. If, moreover, $B$ contains an element $\sigma$ such that for all $\alpha\in B$ and $a\in A$, $\sigma\alpha\hat{a}=\widehat{\alpha (a)}$, then $\varepsilon$ is greatest with this property.\end{proposition}

\section{Total Combinatory Algebras of Partial Functions}\label{partialsection}
With some care, the whole set-up of this paper generalizes to the set ${\rm Ptl}(A,A)$ of partial functions $A\to A$.

For each $\alpha\in {\rm Ptl}(A,A)$ we have a partial function $\varphi _{\alpha}:{\rm Ptl}(A,A)\to A$ given by interrogations. Modifying the definition of sequential functions in such a way that we now consider nontotal sequential trees $T$ and {\em partial\/} maps $F$ from the leaves of $T$ to $A$, we easily see that a partial function ${\rm Ptl}(A,A)\to A$ is sequential, precisely if it is of the form $\varphi _{\alpha}$ for some $\alpha\in {\rm Ptl}(A,A)$. Note that in this case, given $\varphi _{\alpha}$, we can (much simpler than in the proof of \ref{seqcoding}) define the corresponding sequential tree as the set of those finite functions $s$ such that there is an interrogation of $s$ by $\alpha$ which contains all values of $s$ and is in the domain of $\alpha$. Finally, for a leaf $s$ of the tree we we can define $F(s)=b$ if there is an interrogation $u$ of $s$ by $\alpha$ such that $\alpha (u)=\langle r,b\rangle$.

Quite similar to section~\ref{pcasection}, we have a partial combinatory algebra structure on ${\rm Ptl}(A,A)$. This generalizes the construction of $\cal B$ (for $A=\mathbb{N}$) in \cite{OostenJ:casfft,LongleyJ:seqrf}. Just as in Definition~\ref{applicationdef} we have a partial function $\varphi ^a:{\rm Ptl}(A,A)\times {\rm Ptl}(A,A)\to A$ for each $a\in A$, and hence a total function $\alpha ,\beta\mapsto\alpha\beta :{\rm Ptl}(A,A)\times {\rm Ptl}(A,A)\to {\rm Ptl}(A,A)$. The notion of a bisequential function is also straightforward, and analogously to Lemma~\ref{bisequential=representable} we have, for every {\em partial\/} bisequential function $G:{\rm Ptl}(A,A)\times {\rm Ptl}(A,A)\to A$, an element $\phi _G$ such that for all $\alpha$ and $\beta$, $\varphi _{\phi _G\alpha}(\beta )\simeq G(\alpha ,\beta )$.

Again, this is simpler than in the case of total functions: no artificial construction in order to make sure that $\phi _G\alpha$ is a total function, is required. 

The proof of Theorem~\ref{A^Apca} also simplifies, because the elements $k$ and $s$ need not be artificially extended beyond what they have to perform on the relevant interrogations. It follows, that if ${\cal K}_2^p(A)$ is compatible with $A$, $k$ and $s$ can be chosen to be representable in $A$.

Let us write ${\cal K}_2^p(A)$ for the partial combinatory algebra structure on ${\rm Ptl}(A,A)$. Since the application function is total, we speak of a {\em (total) combinatory algebra}.

We have the same map $\gamma :A\to {\cal K}_2^p(A)$ as in Proposition~\ref{AtoA^A}; it is decidable, and every function $A\to A$ is representable w.r.t.\ $\gamma$. It is worth noting that in the partial case, the definition of $\rho$ in the proof of \ref{AtoA^A} can be simplified: we simply define
$$\rho (\langle\langle x,b\rangle ,a\rangle )\;\simeq\;\langle r,\langle r,ab\rangle\rangle$$
and don't need to define $\rho$ outside the set of elements of form $\langle\langle x\rangle\rangle$ or $\langle \langle x,b\rangle\rangle$. It follows, that if ${\cal K}_2^p(A)$ is compatible with $A$, this function $\rho$ is representable in $A$.

The combinatory algebra ${\cal K}_2^p(A)$ satisfies a similar semi-universal property as the one given for ${\cal K}_2(A)$ in Theorem~\ref{factortheorem}, with the map $\gamma :A\to {\cal K}_2^p(A)$, provided ${\cal K}_2^p(A)$ is compatible with $A$. This gives the following corollary.
\begin{corollary}\label{retract1} ${\cal K}_2(A)$ is an elementary sub-pca of ${\cal K}_2^p(A)$ and a retract of it in the category of partial combinatory algebras and isomorphism classes of applicative morphisms.\end{corollary}
\begin{proof} The choice of $k$ and $s$ we made for ${\cal K}_2(A)$ also works for ${\cal K}_2^p(A)$: so the inclusion $i:A^A\to {\rm Ptl}(A,A)$ is elementary; it is also an applicative morphism, and decidable. Furthermore, if we apply the semi-universal property of ${\cal K}_2^p(A)$ to the diagram
$$\diagram A\rto^{\gamma}\dto_{\gamma} & {{\cal K}_2(A)} \\ {{\cal K}_2^p(A)} & \enddiagram$$
we obtain an applicative map $\varepsilon :{\cal K}_2^p(A)\to {\cal K}_2(A)$. Concretely,
$$\varepsilon (\alpha )\; =\; \{\beta\, |\,\text{for all $a\in {\rm dom}(\alpha )$, $\beta\hat{a}=\widehat{\alpha (a)}$}\}$$
It is not hard to show that $\varepsilon i$ is isomorphic to the identity on ${\cal K}_2(A)$.\end{proof}
\medskip

\noindent The pattern, that definitions are simpler and theorems more elegant and smooth in the case of partial functions, extends to the study of sub-pcas of ${\cal K}_2^p(A)$. First of all we have the following proposition:
\begin{proposition}\label{turingjoins} The preorder $\leq _T$ on partial functions $A\to A$ (relative to a partial combinatory algebra structure on $A$) has binary joins.\end{proposition}
\begin{proof} Given partial functions $f$ and $g$ define $f\sqcup g$ by:
$$(f\sqcup g)(y)\;\simeq\; {\sf If}\; p_0y\; {\sf then}\; f(p_1y)\; {\sf else}\; g(p_1y)$$
So $(f\sqcup g)([\top ,x])\simeq f(x)$ and $(f\sqcup g)([\bot ,x])\simeq g(x)$. It is left to the reader that $f\sqcup g$ is a join for $f,g$ with respect to $\leq _T$.\end{proof}
\medskip

\noindent One can now simply define an {\em ideal\/} of ${\cal K}_2^p(A)$ to be a downwards closed set which is also closed under $\sqcup$. Given a subset $B$ of ${\rm Ptl}(A,A)$, let us write $B_{\subseteq}$ for the set
$$B_{\subseteq}\; =\;\{ f\in {\rm Ptl}(A,A)\, |\, \text{there is $g\in B$ such that $f\subseteq g$}\}$$
Furthermore let us write $\bar{a}$ for the partial function $x\mapsto ax$, for $a\in A$.
\begin{proposition}\label{ptlsubpca} Let $A$ be a partial combinatory algebra and suppose ${\cal K}_2^p(A)$ is compatible with $A$. Suppose $B\subset {\rm Ptl}(A,A)$ is closed under the application function.
\begin{rlist}\item If $B$ is downwards closed w.r.t.\ $\leq _T$, then $B$ is an elementary sub-pca of ${\cal K}_2^p(A)$;
\item If for every $a\in A$ we have $\bar{a}\in B$, then $B_{\subseteq}$ is an ideal of ${\cal K}_2^p(A)$.\end{rlist}\end{proposition}
\begin{proof} The first item is the remark made before, that in the partial case we can choose $k$ and $s$ to be representable in $A$.

For the second item: first we remark that there is an element $a\in A$ such that for any $\gamma _1,\gamma _2\in {\rm Ptl}(A,A)$, $\bar{a}\gamma _1\gamma _2=\gamma _1\sqcup\gamma _2$ in ${\cal K}_2^p(A)$. This is left to the reader. So from the hypotheses on $B$ it follows that $B$ is closed under $\sqcup$. Since application in ${\cal K}_2^p(A)$ is monotone in both variables w.r.t.\ $\subseteq$, it follows that also $B_{\subseteq}$ is closed under $\sqcup$.

Next, we see that $B_{\subseteq}$ is downwards closed w.r.t.\ $\leq _T$. Suppose $\beta '\in B$, $\beta\subseteq\beta '$ and $\gamma\leq _T\beta$. We need to show: $\gamma\in B_{\subseteq}$. Since $\gamma\leq _T\beta$, there is an $a\in A$ satisfying $a{\cdot}^{\beta}x=\gamma (x)$ for all $x\in {\rm dom}(\gamma )$. Then also $a{\cdot}^{\beta '}x=\gamma (x)$ for all $x\in {\rm dom}(\gamma )$. But this means that $\gamma\subseteq \bar{a}\beta '$. Since $B$ is closed under application and $\bar{a},\beta '\in B$ by assumption, $\gamma\in B_{\subseteq}$ as desired.\end{proof}
\medskip

\noindent We also have the following generalization of the factorization theorem:
\begin{theorem}\label{ptlfactortheorem} Suppose $B$ is an elementary sub-pca of ${\cal K}_2^p(A)$ containing the (simplified) function $\rho$ of the proof of \ref{AtoA^A} and all constant functions $\hat{a}$. Let $\gamma :A\to B$ be the decidable applicative morphism $\gamma (a)=\{\hat{a}\}$. Then for any decidable applicative morphism $\delta :A\to C$ such that every partial function in $B$ is representable w.r.t.\ $\delta$, there is an applicative morphism $\varepsilon :B\to C$ satisfying $\varepsilon\gamma\cong\delta$.\end{theorem}

Let us apply this to a specific case. For $B$, let us take $T(A)=\{\bar{a}\, |\, a\in A\}$ (where $\bar{a}$ is as in \ref{ptlsubpca}). Since every $\bar{a}$ is representable in $A$, we have a triangle
$$\diagram A\drto_{\gamma}\rrto^{{\rm id}_A} & & A \\  & T(A)\urto_{\varepsilon}\enddiagram$$
which commutes up to isomorphism: $\varepsilon\gamma\cong {\rm id}_A$. About the other composition: $\gamma\varepsilon$, we cannot say much. However, there is another map $\varepsilon ':T(A)\to A$ which makes the triangle commute. Define: $\varepsilon '(\alpha )\; =\; \{ a\, |\, \bar{a}=\alpha\}$. There is an element $b\in A$ such that for every $a\in A$: $\bar{b}\hat{a}=\bar{a}$. This means, that $\gamma\varepsilon '\preceq {\rm id}_{T(A)}$ (the other inequality does not hold). This means that $\gamma\dashv\varepsilon '$ is an adjoint pair of decidable applicative morphisms. From the theory of geometric morphisms between realizability toposes in chapter 2 of \cite{OostenJ:reaics} it follows that there is a geometric surjection: ${\sf RT}(T(A))\to {\sf RT}(A)$. We have obtained the following theorem:
\begin{theorem}\label{covertheorem} Every realizability topos is a quotient of a realizability topos on a total combinatory algebra.\end{theorem}
\medskip

\noindent {\bf Acknowledgement}. Part of the research reported here was carried out by my former Ph.D.\ student Bram Arens.

\begin{small}
\bibliographystyle{plain}

\begin{thebibliography}{10}

\bibitem{BirkedalL:devttc-entcs}
L.~Birkedal.
\newblock Developing theories of types and computability via realizability.
\newblock {\em Electronic Notes in Theoretical Computer Science}, 34:viii+282,
  2000.
\newblock Book version of PhD-thesis.

\bibitem{BirkedalL:relmrr}
{L}. {B}irkedal and {J}. van {O}osten.
\newblock {R}elative and modified relative realizability.
\newblock {\em {A}nn.{P}ure {A}ppl.{L}ogic}, 118:115--132, 2002.

\bibitem{KleeneSC:forrff}
S.C. Kleene.
\newblock {\em Formalized Recursive Functionals and Formalized Relizability},
  volume~89 of {\em Memoirs of the American Mathematical Society}.
\newblock American Mathematical Society, 1969.

\bibitem{LongleyJ:reatls}
J.~Longley.
\newblock {\em Realizability Toposes and Language Semantics}.
\newblock PhD thesis, Edinburgh University, 1995.

\bibitem{LongleyJ:seqrf}
J.~Longley.
\newblock {T}he {S}equentially {R}ealizable {F}unctionals.
\newblock {\em Ann. Pure Appl. Logic}, 117(1--3):1--93, 2002.

\bibitem{mellies:seqass}
Paul-Andr{\'e} Melli{\`e}s.
\newblock Sequential algorithms and strongly stable functions.
\newblock {\em {T}heor.{C}omp.{S}ci.}, 343:237--281, 2005.

\bibitem{PlotkinG:lcfcpl}
G.~Plotkin.
\newblock {LCF} considered as a programming language.
\newblock {\em Theor. Computer Science}, 5:223--255, 1977.

\bibitem{SoareR:recesd}
Robert~I. Soare.
\newblock {\em Recursively Enumerable Sets and Degrees}.
\newblock Perspectives in Mathematical Logic. Springer-Verlag, 1987.

\bibitem{OostenJ:casfft}
J.~van Oosten.
\newblock A combinatory algebra for sequential functionals of finite type.
\newblock In S.B. Cooper and J.K. Truss, editors, {\em Models and
  Computability}, pages 389--406. Cambridge University Press, 1999.

\bibitem{OostenJ:genfrr}
J.~van Oosten.
\newblock {A} general form of relative recursion.
\newblock {\em Notre Dame Journ. Formal Logic}, 47(3):311--318, 2006.

\bibitem{OostenJ:reaics}
J.~van Oosten.
\newblock {\em Realizability: an Introducton to its Categorical Side}, volume
  152 of {\em Studies in Logic}.
\newblock North-Holland, 2008.

\end{thebibliography}

\end{small}
\end{document}